\documentclass[11pt]{amsart} 
\usepackage{amsthm,amsbsy,amsfonts,amssymb,amscd}
\usepackage{bbm}
\usepackage[enableskew]{youngtab}
\usepackage[mathscr]{euscript} 

\usepackage{enumerate} 
\usepackage{tikz} 
\usetikzlibrary{calc,intersections,through, backgrounds,arrows,decorations.pathmorphing,backgrounds,positioning,fit,petri} 
\usetikzlibrary{calc}

\usepackage{tikz-cd} 
\usepackage[all]{xy}
\xyoption{knot}
\xyoption{poly}

\usepackage{hyperref}

\usepackage[margin=1.1in]{geometry}

\newcommand {\supplus}{\mathop{{\subset}\llap{\raise 0.5pt\hbox{\normalfont\small+}\hskip 0.5pt}}}

\newcommand{\dfs}{{/\kern-2pt/}}

\numberwithin{equation}{section}

\newtheoremstyle{notes} {} {} {} {} {\bfseries} {.} {.5em} {}

\theoremstyle{plain}

\newtheorem{prop}{Proposition}

\newtheorem{lemma}{Lemma}

\theoremstyle{remark}

\swapnumbers

\usepackage{etoolbox}

\usepackage{ytableau}
\usepackage{youngtab}

\makeatletter
\pretocmd{\appendix}{\addtocontents{toc}{\protect\addvspace{10\p@}}}{}{}
\makeatother

\theoremstyle{remark}

\newtheoremstyle{construction} {} {} {} {} {\bfseries} { } {0pt} {}

\theoremstyle{construction}

\author[Ivan Penkov]{\;Ivan Penkov}

\address{
Ivan Penkov
\newline Jacobs University Bremen
\newline Campus Ring 1
\newline 28759 Bremen, Germany}
\email{i.penkov@jacobs-university.de}

\usepackage[mathscr]{euscript}
\usepackage{enumerate,color}

\begin{document}
\bibliographystyle{amsplain}

\ifx\useHugeSize\undefined
\else
\Huge
\fi

\relax

\title{ examples of automorphism groups of ind-varieties of generalized flags}
\begin{center}
\it To the memory if Vasil Tsanov
\end{center}

\date{ \today }

\begin{abstract} We compute the automorphism groups of finite and cofinite ind-grassmannians, as well as of the ind-variety of maximal flags indexed by $\mathbb Z_{>0}$. We pay special attention to differences with the case of ordinary flag varieties. 
\end{abstract}
\maketitle


\section{Introduction}
The flag varieties of the classical Lie groups are central objects of study both in geometry and representation theory. In a sense, they are a hub for many directions of research in both fields. Several different infinite-dimensional analogues of the ordinary flag varieties have been studied in the literature, one such analogue being the ind-varieties of generalized flags introduced in \cite{DP} and further investigated in \cite{PT}, \cite{IPW}, \cite{FP1}, \cite{FP2}; see also the survey \cite{IP}. The latter ind-varieties are direct limits of classical flag varieties and are homogeneous ind-spaces for the simple ind-groups $SL(\infty)$, $SO(\infty)$, $Sp(\infty)$. Without doubt, some of these ind-varieties, in particular the ind-grassmannians, have been known long before the paper \cite{DP}. 

A natural question of obvious importance is the question of finding the automorphism groups of the ind-varieties of generalized flags. The purpose of the present note is to initiate a discussion in this direction and to point out some differences with the case of ordinary flag varieties: see Section 4. The topic is very close to Vasil's interests and expertise, and for sure I would have discussed it with him if he were still alive.

{\bf Acknowledgments.} This work has been supported in part by DFG grant PE 980/7-1. 

\section{Automorphisms of finite and cofinite ind-grassmannians}

The base field is $\mathbb C$. Let $V$ be a fixed countable-dimensional complex vector space.  We fix a basis $E=\{e_1,\dots,e_n,\dots\}$ of $V$ and set $V_n:=\operatorname{span}_{\mathbb C}\{e_1,\dots,e_n\}$. Then $V=\cup_n V_n$. Fix $k\in \mathbb{Z}_{>0}$. By definition, $\operatorname{Gr}(k,V)$ is the set of all $k$-dimensional subspaces in $V$ and has an obvious ind-variety structure:
$$\operatorname{Gr}(k,V)=\lim_{\longrightarrow}\operatorname{Gr}(k,V_n).$$

The projective ind-space $\mathbb{P}(V)$ equals $\operatorname{Gr}(1,V)$. Note that the basis $E$ plays no role in this construction. We think  of the ind-varieties $\operatorname{Gr}(k,E)$ for $k\in\mathbb{Z}_{>0}$ as the "finite ind-grassmannians."

The basis $E$ plays a role when defining the "cofinite" ind-grassmannians. Fix a subspace $W\subset V$ of finite codimension in $V$ and such that $E\cap W$ is a basis of $W$. Let $\operatorname{Gr}(W,E,V)$ be the set of all subspaces $W'\subset V$ which have the same codimension in $V$ as $W$ and in addition contain almost all elements of $E$. Then $\operatorname{Gr}(W,E,V)$ has the following ind-variety structure:
$$\operatorname{Gr}(W,E,V)=\lim_{\longrightarrow}\operatorname{Gr} (\operatorname{codim}_V W, \bar V_n)$$
where $\{\bar V_n\}$ is any set of finite-dimensional spaces with the properties that $\bar V_n\supset\operatorname{span}\{E\backslash \{E\cap W\}\}$, $\operatorname{dim}\bar V_n=n>\operatorname{codim}_V W$, $E\cap \bar V_n$ is a basis of $\bar V_n$, and $\cup \bar V_n=V$.

It is clear that the ind-varieties $\operatorname{Gr}(W,E,V)$ and $\operatorname{Gr}(k,V)$ are isomorphic: the isomorphism is given by

\begin{equation}\label{new}\tag{1}
\operatorname{Gr}(W,E,V)\ni W'  \to \operatorname{Ann}W'\subset V_* := \operatorname{span}\{E^*\},
\end{equation}	
where $E^*=\{e_1^*,e_2^*,\dots\}$ is the system of linear functionals dual to the basis $E$: i.e. $e^*_i(e_j)=\delta_{ij}$. The map (\ref{new}) is an obvious analogue of finite-dimensional duality. Therefore the automorphism groups $\operatorname{Aut}\operatorname{Gr}(k,V)$ and $\operatorname{Aut}\operatorname{Gr}(W,E,V)$ for $\operatorname{codim}_W V=k$ are isomorphic; by an automorphism we mean of course an automorphism of ind-varieties.

The following result should in principle be known. We present a proof which shows a connection with the work \cite{PT}.

\begin{prop}
$\operatorname{Aut} \operatorname{Gr}(k,V)=PGL(V)$ where $GL(V)$ denotes the group of all invertible linear operators on $V$ and $PGL(V):= GL(V)/\mathbb{C}_{mult}Id$ (where $\mathbb C_{mult}$ is the multiplicative group of $\mathbb C$).
\end{prop}

\begin{proof} An automorphism $\phi: \operatorname{Gr}(k,V)\to \operatorname{Gr}(k,V)$ induces embeddings $\phi_n: \operatorname{Gr}(k,V_n)\hookrightarrow\operatorname{Gr}(k,V_{N(n)})$ for appropriate $N(n)\geq n$. These embeddings are linear in the sense that \\$\phi^*_n( \mathcal O_{\operatorname{Gr}(k,V_{N(n)})}(1))$ is isomorphic to $\mathcal O_{\operatorname{Gr}(k,V_n)}(1)$, where by $\mathcal{O}_{\cdot}(1)$ we denote the positive generator of the respective Picard group. According to Theorem 1 in \cite{PT}, $\phi_n$ is one of the following:

\begin{enumerate}[(i)]
\item an embedding induced by the choice of an $n$-dimensional subspace $W_n\subset V_{N(n)}$ for some $N(n)\geq n$,
\item an embedding factoring through  a linearly embedded projective space $\mathbb{P}^{M(n)}\subset  \operatorname{Gr}(k, V_{N(n)})$ for some $M(n)<N(n)$.
\end{enumerate}

If $k>2$, option (ii) may hold only for  finitely many $n$ as the contrary implies that the image of $\phi_n$ is contained in a projective ind-subspace $$\mathbb{P}:=\lim_{\longrightarrow} \mathbb{P}^{M(n)}\subset  \operatorname{Gr}(k,V).$$ Then, since $\mathbb{P}$ is not isomorphic to $ \operatorname{Gr}(k,V)$ by Theorem 2 in \cite{PT}, the image of $\phi_n$ would necessarily be a proper ind-subvariety of $ \operatorname{Gr}(k,V)$, which is a contradiction.

For $k=1$, options (i) and (ii) are the same, and therefore without loss of generality we can now assume that for our fixed $k$ option (i) holds for all $n$. The embeddings $\phi_n: \operatorname{Gr}(k,V_n)\hookrightarrow  \operatorname{Gr}(k,V_{N(n)})$ determine injective linear operators $\tilde{\phi_n}:V_n\to V_{N(n)}$. Moreover, the operators $\tilde \phi_n$ are defined up to multiplicative constants which can be chosen so that $\tilde{\phi}_{n}|_{V_{n-1}}=\tilde{\phi}_{n-1}$ for any $n$. Therefore, we obtain a well-defined linear operator
$$\tilde{\phi}: V=\lim_{\longrightarrow} V_n\to V=\lim_{\longrightarrow}V_{N(n)}$$
which induces our automorphism $\phi$. Since $\phi$ is invertible, $\tilde \phi$ is also invertible, and since $\tilde \phi$ depends on a multiplicative constant, we conclude that $\phi$ determines a unique element $ \bar{\phi}\in PGL(V)$.

In this way we have constructed an injective homomorphism
$$ \operatorname{Aut}\operatorname{Gr}(k,V)\to PGL(V),\: \phi \mapsto \bar\phi.$$
The inverse homomorphism $$ PGL(V)\to \operatorname{Aut}\operatorname{Gr}(k,V)$$ is obvious because of the natural action of $PGL(V)$ on $\operatorname{Gr}(k,V)$. The statement follows.
\end{proof}

\section{Ind-variety of  maximal ascending flags}
We now consider a particular ind-variety of maximal generalized flags, in fact the simplest case of maximal generalized flags. Let $V$ and $E$ be as above. Define
$Fl(F_E, E, V)$ as the set of all infinite chains $F_E'$ of subspaces of $V$
$$0\subset(F_E')^1\subset\dots\subset(F'_E)^k\subset\dots$$
where $\operatorname{dim}(F'_E)^k=k$ and $(F'_E)^n= F_E^n:=\operatorname{span}\{e_1,\dots,e_n\}$ for large enough $n$. This set has an obvious structure of ind-variety as
$$Fl(F_E,E,V)=\lim_{\longrightarrow}Fl(F_E^n)$$ 
where $Fl(F_E^n)$ stands for the variety of maximal flags in the finite-dimensional vector space $F_E^n$.

Denote by $GL(E,V)$ the subgroup of $GL(V)$  of automorphisms of $V$ which keep all but finitely many elements of $E$ fixed. The elements of $GL(E,V)$ are the $E$-{\it finitary} automorphisms of $V$.

\begin{prop}
$$\operatorname{Aut}Fl(F_E,E,V)=P(GL(E,V)\cdot B_E)$$
where $B_E\subset GL(V)$ is the stabilizer of the chain $F_E$ in $GL(V)$ and $GL(E,V)\cdot B_E$ is the subgroup of $GL(V)$ generated by $GL(E,V)$ and $B_E$.
\end{prop}

We start with a lemma.

\begin{lemma}
Fix $k\geq2$. Let $\psi_{k-1},$ $ \psi_k:V\to V$ be invertible linear operators such that $\psi_{k-1}(W_{k-1})\subset\psi_k(W_k)$ for any pair of subspaces $W_{k-1}\subset W_k$ of $V$ with $\operatorname{dim}W_{k-1}=k-1$, $\operatorname{dim}W_k=k$. Then $\psi_{k-1}=c\psi_k$ for some $0\neq c\in\mathbb C$.

\end{lemma}

\begin{proof}
Assume the contrary. Let $v$ be a vector in $V$ such that the space  $Z:=\operatorname{span}_{\mathbb{C}}\{\psi_{k-1}(v), \psi_k(v)\}$ has dimension 2. Extend $v$ to a basis $v=v_1,v_2,\dots$ of $V$. Then, setting $W_k=\operatorname{span}_{\mathbb C}\{v_1,\dots, v_k\}$ and $W_{k-1}=\operatorname{span}_{\mathbb C}\{v_1,\dots, v_{k-1}\}$, we see that the condition $\psi_{k-1}(W_{k-1})\subset\psi_k(W_k)$ implies $Z\subset\psi_k(W_k) $. Similarly, setting $W'_{k}=\operatorname{span}_{\mathbb C}\{v_1,v_{k+1}, v_{k+2}\dots,v_{2k-1}\}$ and $W'_{k-1}=\operatorname{span}_{\mathbb C}\{v_1,v_{k+1}, v_{k+2}\dots,v_{2k-2}\}$ we have $Z\subset \psi_k(W'_{k})$.
 However clearly
 $$\operatorname{dim}(W_k \cap W'_k)=1,$$
 hence the dimension of the intersection $\psi_k(W_k)\cap\psi_k(W'_{k})$ must also be $1$ due to the invertibility of $\psi_k$.
Contradiction.

\end{proof} 

\begin{proof}[Proof of Proposition 2.] We first embed $A:=\operatorname{Aut}Fl(F_E,E,V)$ into the group $PGL(V)$. For this we consider the obvious embedding $$A\hookrightarrow\Pi_{i=1}^{\infty}\operatorname{Aut} \operatorname{Gr}(i,V)$$
arising from the diagram of surjective morphisms of ind-varieties

$$
\xymatrix{
&{}Fl(F_E,E,V)\ar[dl]\ar[d]\ar[dr]\\
\mathbb{P}(V)=\operatorname{Gr}(1,V)&{}\operatorname{Gr}(2,V)\:\:\:\:\:\:\:\:\dots&{}\operatorname{Gr}(k,V)\dots\:\:\:.&{}
}
$$

By Proposition 1, the groups $\operatorname{Aut}\operatorname{Gr}(k,V)$ are isomorphic to $PGL(V)$ for all $k\in\mathbb Z_{>0}$. Moreover, it is clear that the injective homomorphism $A\to\Pi_k PGL(V)$ factors through the diagonal of $\Pi_k PGL(V)$ since Lemma $1$ shows that an automorphism from $A$ induces necessarily the same element in $PGL(V)$ via any projection $Fl(F_E,E,V)\to\operatorname{Gr}(k,V)$.

It remains to determine which elements of the group $PGL(V)$ arise as images of elements of $A$. It is clear that this image contains both $PGL(E,V)$ and $PB_E$ as each of these groups acts faithfully on $Fl(F_E,E,V)$. Indeed, the fact that $PGL(E,V)$ acts on $Fl(F_E,E,V)$ is clear. To see that $PB_E$ acts on $Fl(F_E,E,V)$ one notices that  for any $F'_E\in Fl(F_E,E,V)$ and any $\gamma\in PB_E$, the flag $\gamma(F'_E)$ differs from $F_E$ only in finitely many positions, hence is a point on $Fl(F_E,E,V).$ 

On the other hand, it is clear that the image $\bar\phi\in PGL(V)$ of $\phi\in A$ is contained in $P(GL(E,V)\cdot B_E)$. Indeed the composition $\psi\circ\bar\phi$ with a suitable element of $PGL(E,V)$ will fix the point $F_E$ on $Fl(F_E,E,V)$. This means that $\psi\circ\bar\phi\in PB_E$. Therefore the image of $A$ in $PGL(V)$ is contained in $P(GL(E,V)\cdot B_E)$,                  and we are done.

\end{proof}

\section{Discussion}
 First, Proposition $1$ can be generalized to ind-varieties of the form $Fl(F,E,V)$ where $F$ is a finite chain consisting only of finite-dimensional subspaces of $V$, or only of subspaces of finite codimension of $V$. The precise definition of the ind-varieties $Fl(F,E,V)$ is given in \cite{DP}. In these cases, the respective automorphism groups are always isomorphic to $PGL(V)$, however in the case of finite codimension there is a natural isomorphism with $PGL(V_*)$.

We now point out some differences with the case of ordinary flag varieties. A first obvious difference is the following. Despite the fact that $\operatorname{Gr}(k,V)=PGL(E,V)/P_k$, where $P_k$ is the stabilizer in $PGL(E,V)$ of a $k$-dimensional subspace of $V$, the automorphism group of $\operatorname{Gr}(k,V)$ is much larger than $PGL(E,V)$. Therefore $\operatorname{Gr}(k,V)$ is a quotient of any subgroup $G$ satisfying $PGL(E,V)\subset G\subset PGL(V)$, and there is quite a variety of such subgroups. Similar comments apply to the other examples we consider.

Next, we note that the automorphism group of an ind-variety of generalized flags is in general not naturally embedded into $PGL(V)$. Indeed, the case of the cofinite ind-grassmannian $\operatorname{Gr}(W,E,V)$ shows that the natural isomorphism $\operatorname{Aut}\operatorname{Gr}(W,E,V)= PGL(V_*)$ does not embed $\operatorname{Aut}\operatorname{Gr}(W,E,V)$ into $PGL(V)$ by duality, but only embeds $\operatorname{Aut}\operatorname{Gr}(W,E,V)$ into the much larger group $PGL((V_*)^*)$ in a way that its image does not keep the subspace $V\subset(V_*)^*$ invariant. This is  clearly an infinite-dimensional phenomenon.

Finally, recall that the automorphism groups of all flag varieties of the group $GL(n)$ are isomorphic, and inclusions of parabolic subgroups induce isomorphisms of automorphism groups. This note shows that the latter statement is not true for the group $GL(E,V)$  as the injection $\operatorname{Aut}Fl(F_E,E,V)\hookrightarrow \operatorname{Aut}\operatorname{Gr}(k,V)$ constructed in the proof of Proposition $2$ is proper (and both   $\operatorname{Gr}(k,V)$ and $Fl(F_E,E,V)$ are homogeneous ind-varieties for $GL(E,V)$). 

We hope that the above differences motivate a more detailed future study of the automorphism groups of arbitrary ind-varieties of generalized flags.

\end{document}